\documentclass[12pt]{article}
\usepackage{amsmath}
\usepackage{amssymb}
\usepackage{amsthm}
\usepackage{enumitem}
\usepackage{color}
\usepackage{tikz, tikz-cd}
\usepackage{microtype}
\usepackage[lmargin=3cm, rmargin=3cm]{geometry}

\def\QQ{\mathbb{Q}}

\def\b1{{\bf 1}}

\def\oN{{\overline N}}

\newtheoremstyle{upright}  % name
    {}  % Space above
    {}  % Space below
    {\upshape}  % Body font
    {}  % Indent amount
    {\bfseries}  % Theorem head font
    {.}  % Punctuation after theorem head
    {5pt plus 1pt minus 1pt}  % Space after theorem head
    {}  % Theorem head spec (can be left empty, meaning `normal`)

\theoremstyle{upright}

\newtheorem{definition}{Definition}
\newtheorem{theorem}[definition]{Theorem}

\newtheorem{proposition}[definition]{Proposition}
\newtheorem{corollary}[definition]{Corollary}

\newtheorem{remark}[definition]{Remark}

\newcommand{\ZZ}{\mathbb{Z}}

\newcommand{\RR}{\mathbb{R}}
\newcommand{\calX}{\mathcal{X}}
\newcommand{\calN}{\mathcal{N}}
\newcommand{\bfSigma}{\mathbf{\Sigma}}

\title{Birational classification of toric orbifolds}
\author{Johannes Schmitt\footnote{Departement Mathematik, ETH Z\"urich, e-mail: johannes.schmitt@math.ethz.ch}}

\begin{document}

\maketitle

\vspace{-20pt}

\begin{abstract}
We give a complete classification of the torus-equivariant birational equivalence classes of smooth proper toric Deligne--Mumford stacks with trivial generic stabilizer in terms of their associated stacky fans.
\end{abstract}

\section{Introduction}
An \emph{algebraic orbifold} is a smooth separated irreducible Deligne--Mumford stack of finite type over a field $k$ (assumed to be algebraically closed of characteristic zero) with trivial generic stabilizer. 
Following \cite{KT} we say that such orbifolds $\calX, \calX'$ are birationally equivalent if there is a third algebraic orbifold $\widehat \calX$ admitting proper, birational and representable morphisms
\begin{equation} \label{eqn:first_hat}
 \calX \xleftarrow{f} \widehat \calX \xrightarrow{g} \calX'\,.   
\end{equation}
In our paper, we consider the torus-equivariant birational classification problem for \emph{smooth proper toric Deligne--Mumford (DM) stacks} with generically trivial stabilizer (see below for the precise definition). This problem was previously studied by Levchenko in \cite{Levchenko}, who considered the case of dimension $2$ and obtained some birational invariants. In our paper, we give a full classification in arbitrary dimension. 

\section{Main result}
To describe our results, we briefly recall the definition and basic properties of smooth toric DM-stacks.
These were introduced in \cite{BCS} and generalize the canonical cover of simplicial toric varieties (c.f. \cite[Theorem 4.11]{FMN}). 
Similar to the case of normal toric varieties, a smooth toric DM-stack is given by essentially convex-geometric data:
\begin{itemize}
\item a finitely generated abelian group $N$ of rank $d$; we denote the image of the natural map $N \to N_\QQ = N \otimes_{\ZZ} \QQ$ by $\oN$
\item a simplicial rational polyhedral fan $\Sigma$ in $N_\RR = N \otimes_{\ZZ} \RR$,
\item for each ray $\tau \in \Sigma(1)$ a vector $\rho_\tau \in \oN$ which generates $\tau$.
\end{itemize}
Given such a tuple $\mathbf \Sigma = (N, \Sigma, (\rho_\tau)_{\tau \in \Sigma(1)})$, called a \emph{stacky fan}, one constructs a smooth Deligne--Mumford stack $\calX_{\mathbf \Sigma}$ (see \cite{BCS}). The coarse moduli space of $\calX_{\mathbf \Sigma}$ is the toric variety associated to the fan $\Sigma$. We have that $N = \overline N$ is torsion-free if and only if $\calX_{\mathbf \Sigma}$ is an {orbifold}, i.e. has generically trivial stabilizer (see \cite[Lemma 7.15]{FMN}). Under this assumption, for $\sigma \in \Sigma$ we denote by $N_\sigma \subseteq N$ the sub-lattice spanned by the $\rho_\tau$ for $\tau \in \sigma(1)$. The index of $N_\sigma$ in $N \cap \mathrm{Span}_\QQ \sigma$ is the order of the generic stabilizer group on the torus-invariant stratum of $\calX_{\mathbf \Sigma}$ associated to $\sigma$.
We remark for later use that for a face $\pi \subseteq \sigma$ we have $N_\pi = N_\sigma \cap \mathrm{Span}_\QQ \pi$.

From now on, we only work with toric orbifolds $\calX_\bfSigma$, i.e. smooth toric DM-stacks with generically trivial stabilizer. As for toric varieties, they contain the algebraic torus $T = T_{\calX_\bfSigma} = \mathrm{Spec} (k[N^\vee])$ as an open dense substack. Following \cite{Levchenko}, we say that two such orbifolds $\calX, \calX'$ are $T$-equivariantly birationally equivalent if there exists a diagram \eqref{eqn:first_hat} such that $\widehat \calX$ is 
an orbifold with an action of $T$ and a dense equivariant embedding of $T$ such that the maps $f,g$ are $T$-equivariant.\footnote{By composing $f,g$ with the action of suitable points in $T(k)$ on $\calX, \calX'$, we can assume without loss of generality these maps restrict to the identity on $T$.}
It follows from \cite[Theorem 7.17]{FMN} that such an orbifold $\widehat \calX$ is a toric orbifold in the sense described above, i.e. given by $\widehat \calX = \calX_{\widehat{\bfSigma}}$ for a stacky fan $\widehat{\bfSigma} = (N, \widehat \Sigma, (\widehat \rho_\tau)_{\tau \in \widehat \Sigma(1)})$.

Given two toric DM-stacks $\calX_{\mathbf \Sigma'}, \calX_{\mathbf \Sigma}$ containing the same torus $T$, the identity on $T$ induces a $T$-equivariant birational map $\calX_{\mathbf \Sigma'} \dashrightarrow \calX_{\mathbf \Sigma}$. The following proposition gives a criterion when this map is a morphism, respectively a representable morphism, in terms of the stacky fans $\mathbf{\Sigma}, \mathbf{\Sigma}'$.

\begin{proposition} \label{Prop:morphism}
Let $\mathbf{\Sigma} = (N, \Sigma, (\rho_\tau)_{\tau \in \Sigma(1)})$ and $\mathbf \Sigma' = (N, \Sigma', (\rho'_{\tau'})_{\tau' \in \Sigma'(1)})$ be two stacky fans with the same underlying group $N$, such that $\Sigma, \Sigma'$ are full-dimensional\footnote{Here we say that a rational polyhedral fan in $N_\RR \cong \RR^d$ is full-dimensional if each cone of the fan is a face of a $d$-dimensional cone inside the fan. This condition is automatic for complete fans.} with the same support $|\Sigma| = |\Sigma'|$. Then
the birational map $\calX_{\mathbf \Sigma'} \dashrightarrow \calX_{\mathbf \Sigma}$ extending the identity of the torus $T = \mathrm{Spec} (k[N^\vee])$ is
\begin{itemize}
\item[(a)] a morphism, if and only if the identity on $N_\RR$ induces a map of fans $\Sigma' \to \Sigma$ such that for each $\tau' \in \Sigma'(1)$ and any cone $\sigma \in \Sigma$ containing $\tau'$, the element $\rho'_{\tau'}$ is an integral linear combination of the $\rho_\tau$ for $\tau \in \sigma(1)$,
\item[(b)] a representable morphism if and only if the identity on $N_\RR$ induces a map of fans $\Sigma' \to \Sigma$ and for any full-dimensional cones $\sigma \in \Sigma(d)$, $\sigma' \in \Sigma'(d)$ with $\sigma' \to \sigma$, we have an equality of sublattices $N_\sigma = N_{\sigma'}$. 
\end{itemize}
\end{proposition}
\begin{proof}
The fact that the fan-theoretic condition in (a) is sufficient is Remark 4.5 in \cite{BCS}, and the converse implication follows from \cite[Theorem 3.4]{StackyFans}. For part (b) we note that for any cone $\sigma \in \Sigma$, the generic stabilizer group of the associated torus-invariant subset of $\calX_\bfSigma$ is given by $N \cap \mathrm{Span}_\QQ \sigma / N_\sigma$ (as follows from \cite[Proposition 4.3]{BCS}). By part (a), we obtain a morphism if and only if $N_{\sigma'} \subseteq N_{\sigma}$, and the injectivity of stabilizer groups for the representability of this morphism is then equivalent to the other inclusion $N_{\sigma'} \supseteq N_{\sigma}$. We use here that it is sufficient to check the condition at the maximal cones since for a face $\pi \subseteq \sigma$ we have $N_\pi = N_\sigma \cap \mathrm{Span}_\QQ \pi$.
\end{proof}

Using this proposition, we see that fixing $\mathbf \Sigma$, the torus-equivariant morphisms $\calX_{\mathbf \Sigma'} \to \calX_{\mathbf \Sigma}$ are precisely induced by the choices of a simplicial subdivision $\Sigma' \to \Sigma$ and ray generators $(\rho'_{\tau'})_{\tau' \in \Sigma'(1)}$ such that condition (a)  is satisfied. Similarly, the representable morphisms correspond to subdivisions such that (b) is satisfied (in this case the ray generators are uniquely determined by condition (b)).

\begin{theorem} \label{thm:main}
Let $\calX_{\bfSigma}$ and $\calX_{\bfSigma'}$ be two proper $d$-dimensional toric orbifolds together with an identification $T = T_{\calX_{\bfSigma}} = T_{\calX_{\bfSigma}}$ of their tori. Then $\calX_{\bfSigma}$ and $\calX_{\bfSigma'}$ are $T$-equivariantly birationally equivalent if and only if for any $\sigma \in \Sigma(d)$ and $\sigma' \in \Sigma'(d)$ such that the interiors of $\sigma$ and $\sigma'$ intersect, we have $N_\sigma = N_\sigma'$.
\end{theorem}
\begin{proof}
By definition, we have that $\calX_{\bfSigma}$ and $\calX_{\bfSigma'}$ are $T$-equivariantly birationally equivalent if and only if we can find a toric orbifold $\widehat \calX$ (with torus $T$), together with representable, proper $T$-equivariant birational morphisms
\begin{equation} \label{eqn:birational_hat}
\begin{tikzcd}
& \widehat \calX \arrow[dl, "f", swap] \arrow[dr, "g"] & \\
\calX_{\bfSigma} & & \calX_{\bfSigma'}
\end{tikzcd}\
\end{equation}
such that $f,g$ restrict to the identity on $T$.
As mentioned before, it follows from \cite[Theorem 7.17]{FMN} that the orbifold $\widehat \calX$ is given by $\widehat \calX = \calX_{\widehat{\bfSigma}}$ for a stacky fan $\widehat{\bfSigma} = (N, \widehat \Sigma, (\widehat \rho_\tau)_{\tau \in \widehat \Sigma(1)})$. 

All toric orbifolds  in the diagram \eqref{eqn:birational_hat} are proper, so their fans are full-dimensional with support $N_\RR$. By Proposition \ref{Prop:morphism}, the fan $\widehat \Sigma$ is a refinement of $\Sigma, \Sigma'$. For $\sigma \in \Sigma(d)$ and $\sigma' \in \Sigma'(d)$ with overlapping interior, choose a cone $\widehat \sigma \in \widehat \Sigma(d)$ whose interior maps to $\mathrm{int}(\sigma) \cap \mathrm{int}(\sigma')$. Then the fact that $f,g$ are representable implies $N_\sigma = N_{\widehat \sigma} = N_{\sigma'}$ by Proposition \ref{Prop:morphism} (b).

Conversely, assume $N_\sigma = N_\sigma'$ for any $\sigma \in \Sigma(d)$ and $\sigma' \in \Sigma'(d)$ such that the interiors of $\sigma$ and $\sigma'$ intersect. 
Then we claim that there exists an iterated stacky star subdivision\footnote{This is a variant of the usual star subdivision of a fan along a cone which takes into account the chosen ray generators $\rho_\tau$, see \cite{EdidinMore} for the definition.} $\widehat \bfSigma$ of $\bfSigma$ such that $\widehat \Sigma$ is a refinement of $\Sigma'$. 
The analogous statement for non-stacky fans is proven in \cite{DCPII} (see the Lemmas in Sections 2.2 and 2.3 and the Theorem in Section 2.4 of \cite{DCPII}). The proof is effective, describing a procedure for choosing the sequence of cones to subdivide. Following the same algorithm line-by-line and replacing the primitive generators of rays with the chosen generators $\rho_\tau$ gives the desired result for stacky fans.
We claim that taking $\widehat \calX = \calX_{\widehat \bfSigma}$, we obtain a diagram \eqref{eqn:birational_hat} of proper representable morphisms as desired. Indeed, for the conditions of Proposition \ref{Prop:morphism} (b), we already checked that that the underlying fan $\widehat \Sigma$ of $\widehat \bfSigma$ refines $\Sigma, \Sigma'$. On the other hand, the map $f$ induced by iterated stacky star subdivision is representable, so for a cone $\widehat \sigma \in \widehat \Sigma(d)$ mapping to $\sigma \in \Sigma(d)$ and $\sigma' \in \Sigma'(d)$, we have $N_{\widehat \sigma} = N_\sigma$. But then $\mathrm{int}(\sigma) \cap \mathrm{int}(\sigma') \supseteq \mathrm{int}(\widehat \sigma) \neq \emptyset$, so by assumption $N_{\sigma'} = N_{\sigma} = N_{\widehat \sigma}$ as desired.
\end{proof}
As a consequence of the above result, we can give a complete classification of the birational equivalence classes of proper toric orbifolds. To state it, we introduce the following notion (which we did not find in this precise shape in the literature).
\begin{definition} \label{def:conical_partition}
A \emph{conical polyhedral partition} of $N_\RR$ is a finite collection $(C_i)_{i \in I}$ of non-empty subsets $\emptyset \neq C_i \subseteq N_\RR$ such that
\begin{itemize}
\item each $C_i$ is a finite union of full-dimensional rational polyhedral cones,
\item the union of the $C_i$ is all of $N_\RR$,
\item the interiors of the $C_i$ are pairwise disjoint.
\end{itemize}
\end{definition}
We claim that for such a conical polyhedral partition, one can in fact find a fan $\Sigma_0$ with support $N_\RR$ such that each of the $C_i$ is a union of cones in $\Sigma_0(d)$. Indeed, one way to obtain $\Sigma_0$ is to subdivide $N_\RR$ at each defining hyperplane of each of the full-dimensional rational polyhedral cones used to cover the sets $C_i$. Conversely, given a complete fan $\Sigma_0$ and a partition $I$ of $\Sigma(d)$ into non-empty subsets, the sets $C_i = \bigcup_{\sigma \in I} \sigma$ form a conical polyhedral partition.

\begin{definition}
A \emph{sublattice coloring} of $N$ is a conical polyhedral partition $(C_{N'})_{N' \in \mathcal{N}}$ indexed by a finite set $\mathcal{N}$ of finite-index sublattices $N' \subseteq N$ such that for any $N', N'' \in \mathcal{N}$, we have $C_{N'} \cap C_{N''} \cap N' = C_{N'} \cap C_{N''} \cap N''$.
% The sublattice coloring is called \emph{minimal} if for any $i \neq j$ such that $C_i, C_j$ meet in codimension $1$ (i.e. such that $C_i \cap C_j$ contains a cone of dimension $d-1$) we have $N_i \neq N_j$.
\end{definition}
Given a stacky fan $\bfSigma$ with an underlying fan $\Sigma$ which is complete, we define 
\[\calN = \calN(\bfSigma) =\{N_\sigma : \sigma \in \Sigma(d)\}\]
as the set of lattices associated to maximal cones of $\Sigma$, and given such an $N' \in \calN$ we denote by
\[
C_{N'} = \bigcup_{\substack{\sigma\in \Sigma(d):\\N_\sigma = N'}} \sigma
\]
the associated union of maximal cones $\sigma$ with lattice $N_\sigma = N'$. We claim that $(C_{N'})_{N' \in \calN}$ is a sublattice coloring of $N$. Indeed, by the remark below Definition \ref{def:conical_partition} we see that $(C_{N'})_{N'}$ is a conical polyhedral partition of $N_\RR$. Thus the only non-trivial property to check is $C_{N'} \cap C_{N''} \cap N' = C_{N'} \cap C_{N''} \cap N''$. It follows from the fact that $C_{N'} \cap C_{N''}$ is a union of cones $\pi \in \Sigma$ combined with the previous observation that for $\pi$ a face of $\sigma \in \Sigma$ we have $N_\pi = N_\sigma \cap \mathrm{Span}_\QQ \pi$.

% an equivalence relation $\sim$ on the maximal cones $\sigma \in \Sigma(d)$ generated by declaring $\sigma \sim \sigma'$ if $N_\sigma = N_{\sigma'}$ and $\sigma \cap \sigma'$ is a facet (i.e. face of dimension $d-1$) of both $\sigma, \sigma'$. 
% Let $I = \Sigma(d)/\sim$ be the set of equivalence classes, and for $i \in I$ let $C_i = \bigcup_{\sigma \in i} \sigma$ be the union of cones in the equivalence class $i$. Then it's easy to see that the $C_i$ form a conical polyhedral partition and that setting $N_i = N_{\sigma_i}$ for any representative $\sigma_i \in i$ we obtain a sublattice coloring $((C_i)_i, (N_i)_i)(\bfSigma)$ of $N$.

\begin{corollary} \label{cor:classification}
The $T$-equivariant birational equivalence classes of proper toric orbifolds $\calX_\bfSigma$ with torus $T = \mathrm{Spec}(k[N^\vee])$ are in bijection to sublattice colorings of $N$ by sending the class $[\calX_\bfSigma]$ to the sublattice coloring $(C_{N'})_{N' \in \calN(\bfSigma)}$ above.
\end{corollary}
\begin{proof}
Theorem \ref{thm:main} implies that two proper toric orbifolds are birationally equivalent if and only if they have the same associated sublattice coloring. This shows that their birational equivalence classes inject into the set of sublattice colorings. 

Conversely, given such a sublattice coloring $(C_{N'})_{N' \in \mathcal{N}}$ choose a fan $\Sigma_0$ such that each $C_{N'}$ is a union of maximal cones of $\Sigma_0$. By performing a barycentric subdivision, we can assume without loss of generality that $\Sigma_0$ is simplicial. For each cone $\sigma \in \Sigma_0$ choose $N' \in \calN$ such that $\sigma \subseteq C_{N'}$ and define $N_\sigma = \mathrm{Span}(\sigma) \cap {N'}$. By the definition of sublattice coloring, this sublattice $N_\sigma \subseteq N$ is independent of the choice of the set $C_{N'}$ containing $\sigma$. 
We would like to construct a stacky fan $\bfSigma_0'$ whose underlying fan $\Sigma_0'$ refines $\Sigma_0$ and such that the lattices on the cones of $\Sigma_0'$ are the restrictions of the lattices $N_\sigma$.

To obtain $\bfSigma_0'$, observe that associated to each cone $\sigma$ of $\Sigma_0$ we have two lattices: the lattice $N_\sigma$ above, and its sublattice $N_\sigma'$ spanned by the primitive generators $\rho_\tau \in \tau \cap N_\tau \cong \mathbb{N}$ of the rays $\tau \in \sigma(1)$. We define the multiplicity $\mathrm{mult}(\sigma)$ as the index of $N_\sigma'$ in $N_\sigma$. 
%If $\mathrm{mult}(\sigma)=1$ for all $\sigma \in \Sigma_0$, it follows that the desired stacky fan is just $\Sigma_0$ with the ray generators $\rho_\tau$ for $\tau \in \Sigma_0(1)$. 
If some of the multiplicities are strictly greater than $1$, we can perform iterated star subdivision on $\Sigma_0$ to reduce their multiplicities, as explained in \cite[Section 8.2]{MR0495499} (adapting the procedure explained there to use the lattices $N_\sigma$ for operations on the cone $\sigma$). After finitely many steps we arrive at a simplicial refinement $\Sigma_0'$ of $\Sigma_0$ with multiplicity $1$ on each of the cones, giving the desired stacky fan $\bfSigma_0' = (N, \Sigma_0', (\rho_{\tau'})_{\tau' \in \Sigma_0'(1)})$.
\end{proof}

\begin{remark}
The proofs of Theorem \ref{thm:main} and Corollary \ref{cor:classification} featured many combinatorial and convex geometric operations (like the multiplicity reduction above). A more conceptual and geometric approach is possible via the theory of possibly singular toric DM-stacks and their associated \emph{KM fans} in the sense of \cite{GillamMolcho}. Here a (lattice) KM fan is a triple $F=(N, \Sigma, (N_\sigma)_{\sigma \in \Sigma})$ of a rational polyhedral fan $\Sigma$ in the lattice $N$ together with choices of sublattices $N_\sigma \subseteq N$ associated to its cones which are finite index in $\mathrm{Span}(\sigma) \cap N$ and compatible under face inclusions.
Associated to $F$, the paper \cite{GillamMolcho} defines a toric, but possibly singular DM-stack $\calX(F)$. Any stacky fan $\bfSigma$ defines a KM fan $F$ by taking $N_\sigma$ to be the lattice spanned by the chosen ray generators $\rho_\tau$ of $\sigma$ as before. The constructed toric DM-stacks $\calX(\bfSigma)$ and $\calX(F)$ agree and for a KM fan $F$, we have that $\calX(F)$ is smooth if and only if $F$ comes from a stacky fan.

With this formalism in mind, we can sketch a geometric interpretation of the proofs presented above. Indeed, given complete stacky fans $\bfSigma, \bfSigma'$ with compatible lattices on their maximal cones as in Theorem \ref{thm:main}, choose any common refinement $\Sigma''$ of their underlying fans. Then the cones $\sigma''$ of $\Sigma''$ carry natural sublattices $N_{\sigma''}$ of $N \cap \mathrm{Span}(\sigma)$ induced from the associated lattices of their coarsenings in $\Sigma, \Sigma'$ (which are compatible by assumption). Denote by $F''=(N, \Sigma'', (N_{\sigma''})_{\sigma''})$ the associated KM fan. 
By functoriality, the associated toric DM-stack $\calX(F'')$ admits toric maps to $\calX_\bfSigma, \calX_{\bfSigma'}$ which are representable by \cite[Theorem 3.11.2]{GillamMolcho}.
Composing these with a toric resolution of singularities $\widehat \calX \to \calX(F'')$ gives the diagram \eqref{eqn:birational_hat} showing that $\calX_\bfSigma, \calX_{\bfSigma'}$ are $T$-equivariantly birationally equivalent.

Similarly, in Corollary \ref{cor:classification} given a sublattice coloring $(C_{N'})_{N'}$ choose a fan $\Sigma_0$ such that each $C_{N'}$ is a union of maximal cones of $\Sigma_0$.
Given a cone $\sigma \in \Sigma_0$ contained in some $C_{N'}$, we obtain the lattice $N_\sigma = N' \cap \mathrm{Span}(\sigma)$ (which is independent of the choice of $C_{N'}$ containing $\sigma$ by the definition of sublattice colorings). Again we obtain a KM fan $F=(N, \Sigma_0, (N_\sigma)_\sigma)$, and we can choose any toric resolution $\widehat \calX$ of $\calX(F)$, whose birational equivalence class then induces the given sublattice coloring  $(C_{N'})_{N'}$. 

\end{remark}

% \begin{remark}
% An alternative definition of birational equivalence of toric orbifolds $\calX_\bfSigma, \calX_{\bfSigma'}$ would be to ask them to be $T$-equivariantly birationally equivalent with respect to \emph{some} identification $T_{\calX_\bfSigma} \cong T \cong T_{\calX_{\bfSigma'}}$.
% The corresponding variant of Theorem \ref{thm:main} then says that $\calX_\bfSigma$ and $\calX_{\bfSigma'}$ are birationally equivalent in this weaker sense if there exists an element $A \in \mathrm{GL}(N) \cong \mathrm{GL}(d, \mathbb{Z})$ such that the fans $A \cdot \bfSigma$ and $\bfSigma'$ satisfy the assumptions in Theorem \ref{thm:main}, where $A$ acts in a natural way on the underlying fan $\Sigma$ of $\bfSigma$ and its chosen ray generators $\rho_\tau$. In this way, e.g. $\mathbb{P}^1$ with a $B \mathbb{Z}/r\mathbb{Z}$-orbifold point at $0$ is birationally equivalent to $\mathbb{P}^1$ with such an orbifold point at $\infty$.
% \end{remark}

\section*{Acknowledgements}
Our deep gratitude goes to Andrew Kresch, for his guidance to the literature on toric orbifolds, numerous pieces of advice on technical challenges related to stacky fans and detailed feedback on a preliminary version of the present article, improving exposition and clarity of our arguments. We also thank Sam Molcho for pointing out the theory of KM fans and helpful suggestions concerning toric resolution of singularities.

\bibliographystyle{alpha}
\bibliography{main}

%\vspace{+16 pt}
%\noindent Andrew Kresch \\
%\noindent Mathematical institute \\
%\noindent University of Zurich \\
%\noindent ??
%
% \vspace{+16 pt}
% %\noindent Johannes Schmitt \\
% \noindent Departement Mathematik, ETH Z\"urich \\
% \noindent johannes.schmitt@math.ethz.ch  

\end{document}